\documentclass[a4paper,12pt]{amsart}
\usepackage{ae,amsfonts,euscript,enumerate}
\usepackage{amsmath,euscript}
\usepackage{amssymb}
\usepackage{amsthm}
\usepackage{enumerate}
\usepackage{amsfonts}
\usepackage{comment}
\usepackage{colortbl}
\usepackage{a4wide}
\usepackage{amssymb,amsmath,array}

\newtheorem{thm}{Theorem}[section]
\newtheorem{lem}[thm]{Lemma}

\newtheorem{prop}[thm]{Proposition}

\newtheorem{conj}[thm]{Conjecture}

\theoremstyle{defn}
\newtheorem{defn}{Definition}

\newtheorem{rem}{Remark}

\newtheorem*{notn}{Notation}

\newcommand{\Ca} {\mathcal{C}_{m,n}}

\newcommand{\gray}{\cellcolor[gray]{0.1} }

\newcommand{\gc}{ [ \hspace{-0.65mm} [}
\newcommand{\dc}{]  \hspace{-0.65mm} ]}

\newcommand{\ia}{i,\alpha}

\newcommand{\hc}{\mathcal{H}}

\newcommand{\U}{U_q^+(\g)}

\newcommand{\spec}{{\rm Spec}}

\newcommand{\prim}{{\rm Prim}}

\newcommand{\g}{\mathfrak{g}}

\newcommand{\comp}{\mathbb{C}}

\def\C{\mathbb{K}}

\def\co{{\mathcal O}}

\def\oqmtwo{\co_q(M_2)}
\def\oqmm13{\co_q(M_{1,3})}
\def\oqm23{\co_q(M_{2,3})}
\def\oqmmn{\co_q(M_{m,n})}
\def\oqmnm{\co_q(M_{n,m})}

\def\ia{i,\alpha}

\input{xy}
\xyoption{all}

\def\eqref#1{(\ref{#1})}

\usepackage[dvips]{graphicx}
\usepackage{amssymb}
\usepackage{amsmath}
\usepackage[latin1]{inputenc}
\newif\ifpdf
       \ifx\pdfoutput\undefined
       \pdffalse 
       \else
       \pdfoutput=1 
       \pdftrue
\fi

\ifpdf
       \else
       \usepackage{graphicx,color,rotating}
\fi
\usepackage{fullpage}
\usepackage{amsmath,euscript}
\usepackage{amssymb}
\usepackage{amsthm}
\usepackage{enumerate}
\usepackage{amsfonts}
\usepackage{comment}
\usepackage{colortbl}
\usepackage{a4wide}
\usepackage{amssymb,amsmath,array}

\input{xy}
\xyoption{all}

\title[]{On the dimension of $H$-strata in quantum algebras}
\author{J.~Bell and S.~Launois}
\thanks{The first named author thanks NSERC for its generous support.  The second named author research was supported by a Marie Curie European Reintegration Grant within the $7^{\mbox{th}}$ European Community Framework Programme}

\address{Jason Bell\\
Department of Mathematics\\
Simon Fraser University\\
Burnaby, BC V5A 1S6, Canada
}

\email{jpb@math.sfu.ca}

\address{
St\'ephane Launois \\
School of Mathematics, Statistics \& Actuarial science \\
University of Kent\\
Canterbury, Kent CT2 7NF, United Kingdom}

\email{s.launois@kent.ac.uk}

\begin{document}
\bibliographystyle{plain}

\begin{abstract} We study the topology of the prime spectrum of an algebra supporting a rational torus action. More precisely, we study inclusions between prime ideals that are torus-invariant using the $H$-stratification theory of Goodearl and Letzter on one hand and the theory of deleting derivations of Cauchon on the other. We also give a formula for the dimensions of the $H$-strata described by Goodearl and Letzter. We apply the results obtained to the algebra of $m \times n$ generic quantum matrices to show that the dimensions of the $H$-strata 
are bounded above by the minimum of $m$ and $n$, and that moreover all the values between $0$ and this bound are achieved.
\end{abstract}
\maketitle

\section{Introduction} 
\label{int}

We denote by $R=\oqmmn$ the standard quantization of the ring of
regular functions on $m \times n$ matrices with entries in a field $\mathbb{K}$; it is
the $\mathbb{K}$-algebra generated by the $m \times n $ indeterminates
$Y_{\ia}$, $1 \leq i \leq m$ and $ 1 \leq \alpha \leq n$, subject to the
following relations:\\ \[
\begin{array}{ll}
Y_{i, \beta}Y_{i, \alpha}=q^{-1} Y_{i, \alpha}Y_{i ,\beta},
& (\alpha < \beta); \\
Y_{j, \alpha}Y_{i, \alpha}=q^{-1}Y_{i, \alpha}Y_{j, \alpha},
& (i<j); \\
Y_{j,\beta}Y_{i, \alpha}=Y_{i, \alpha}Y_{j,\beta},
& (i <j,  \alpha > \beta); \\
Y_{j,\beta}Y_{i, \alpha}=Y_{i, \alpha} Y_{j,\beta}-(q-q^{-1})Y_{i,\beta}Y_{j,\alpha},
& (i<j,  \alpha <\beta), 
\end{array}
\]
where $q \in \mathbb{K}^*$ is not a root of unity. We note that the torus $(\mathbb{K}^*)^{m+n}$ acts on $R$ by $\mathbb{K}$-algebra automorphisms via the action
$$(a_1,\dots,a_m,b_1,\dots,b_n)\cdot Y_{\ia} = a_i b_\alpha Y_{\ia} \quad {\rm
for~all} \quad \: (\ia)\in \gc 1,m \dc \times \gc 1,n \dc.$$ 

Understanding this torus action has been responsible for most of the important advances that have been made in the study of quantum matrices.  The most important object of study is the prime spectrum of $R$.  In analogy with algebraic geometry, where great understanding of commutative rings comes from the study of their prime spectra, one seeks to understand the prime spectrum of $R$ and its topology.  The noncommutativity introduced by the parameter $q$ in quantum matrices makes the prime spectrum of $R$ harder to understand than the prime spectrum of the coordinate ring of the variety of $m\times n$ matrices and much work has been done in understanding the structure of this topological object. The most important of these advances is the stratification theory of Goodearl and Letzter \cite{gl2}.

To describe the work of Goodearl and Letzter, we give a few basic definitions.
Let $A$ be a $\mathbb{K}$-algebra with a group $H$ acting on it by $\mathbb{K}$-algebra automorphisms. A two-sided ideal $I$ of $A$ is said to be {\em $H$-invariant} if $h\cdot I=I$ for all
$h \in H$.  An {\em $H$-prime ideal} of $A$ is a proper $H$-invariant ideal
$J$ of $A$ such that whenever $J$ contains the product of two
$H$-invariant ideals of $A$, $J$ contains at least one of them. We denote
by $H$-$\spec(A)$ the set of all $H$-prime ideals of $A$. Observe
that if $P$ is a prime ideal of $A$ then
\begin{equation}
(P:H)\ := \ \bigcap_{h\in H} h\cdot P
\end{equation}
 is an $H$-prime ideal of
$A$. This observation allowed Goodearl and Letzter \cite{gl2}
(see also \cite{bgbook}) to construct a
stratification of the prime spectrum of $A$ that is indexed by the
$H$-spectrum. Indeed, let $J$ be an $H$-prime ideal of $A$. We denote
by $\spec_J (A)$ the {\em $H$-stratum} associated  to $J$; that is,  
\begin{equation}
\spec_J (A)=\{ P \in \spec(A) \mbox{ $\mid$ } (P:H)=J \}.
\end{equation}
Then the $H$-strata of $\spec(A)$ form a partition of $\spec(A)$
\cite[Chapter II.2]{bgbook}; that
is,
\begin{equation}
\label{eq:Hstratification}
 \spec(A)= \bigsqcup_{J \in H\mbox{-}\spec(A)}\spec_J(A).
 \end{equation}
This partition is the so-called {\em $H$-stratification} of $\spec(A)$. 

When the $H$-spectrum of $A$ is finite this
partition is a powerful tool in the study of the prime spectrum of $A$. 

As we work in the generic case where $q$ is not a root of unity, the ring $R$ of $m\times n$ quantum matrices has a finite  $\hc=(\mathbb{K}^*)^{m+n}$-spectrum.  Remarkably, for each $\hc$-prime $J$, the space $\spec_J(R)$ is homeomorphic to ${\rm Spec}(\mathbb{K}[z_1^{\pm 1},\ldots ,z_d^{\pm 1}])$ for some $d$ which depends on $J$.  This $d$ is simply the (Krull) dimension of the $\hc$-stratum $\spec_J(R)$.

The work of Goodearl and Letzter spurred much research into the structure of $\spec(R)$ in terms of the $\hc$-spectrum.  Some of the main themes in the study of the $\hc$-spectrum have been to compute its size, to compute the structure of the poset of $\hc$-primes under inclusion, and to compute the dimensions of the $\hc$-strata and how they are distributed.
  
The question of the size of the $\hc$-spectrum of $R$ was answered by Cauchon \cite{c2}.
For many years the finiteness of the $\hc$-spectrum of $R$ was 
known, but no formula for its size was known---except for small values of $m$ and $n$---due to the complicated nature of the relations in $R$.   Cauchon \cite{c2} used his theory of deleting derivations to compute the size of the $\hc$-spectrum of $R$.   In particular, the set of $\hc$-primes is in $1$-$1$ correspondence with a set of combinatorial objects called Cauchon diagrams.  In fact, Cauchon's method applies to a much broader class of algebras, the so-called CGL extensions (CGL stands for Cauchon-Goodearl-Letzter), and the term Cauchon diagram has now acquired a more general meaning than the one we now describe for quantum matrices.

\begin{defn} {\em An $m\times n$ \emph{Cauchon diagram} $C$ is simply an $m\times n$ grid consisting of $mn$ boxes in which certain boxes are coloured black.  We require that the collection of black boxes have the following property:
\vskip 2mm
\noindent If a box is black, then either every box strictly to its left is black or every box strictly above it is black.\\
We let $\Ca$ denote the collection of $m\times n$ Cauchon diagrams.} \label{cdef}
\end{defn}
\begin{figure}
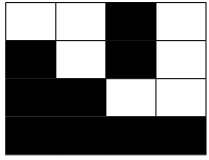

\begin{tabular}{|p{0.30cm}|p{0.30cm}|p{0.30cm}|p{0.30cm}|}

\hline
 &  & \gray &  
 \\
\hline
\gray & & \gray &  
 \\
\hline
\gray & \gray & &  \\
\hline
\gray & \gray & \gray & \gray  \\
\hline

\end{tabular}
\caption{ An example of a $4\times 4$ Cauchon diagram.}
\label{ex-figure}
\end{figure}

Cauchon \cite{c2} showed that the $\hc$-primes of the ring $R$ of $m\times n$ quantum matrices are paramaterized by the collection of $m\times n$ Cauchon diagrams and he also gave a closed formula for the size of this set.  Moreover it is known that the poset $\hc$-$\spec(R)$ 
(under inclusion) is isomorphic to a subposet of the symmetric group $S_{m+n}$ endowed with the Bruhat order \cite{lau}.

Some of the major questions that remain are to determine the possible dimensions of $\hc$-strata that can occur in $R$ and to give a formula for the dimension of a stratum in terms of the associated Cauchon diagram. We answer these questions. In particular, we prove the following result.

\begin{thm} Let $m$ and $n$ be natural numbers.  Then the dimensions of $\hc$-strata in $\oqmmn$ are all at most $\min(m,n)$; moreover, for each $d\in \{0,1,\ldots, \min(m,n)\}$ there exists a $d$-dimensional $\hc$-stratum.
\label{intro-bound}
\end{thm}

Previously, the best known bound for the dimensions of $\hc$-strata in $\oqmmn$ was $m+n-1$, so this result represents a significant improvement. 

Regarding the dimension of the stratum associated to a given Cauchon diagram, we give a formula which only relies on the Cauchon diagram, 
see Proposition \ref{dimstratum}. In fact, we are able to give a formula for a much broader class of algebras, called \emph{uniparameter CGL extensions}. This class of algebras inlcudes in particular the so-called quantum Schubert cells $U_q[w]$ defined by De Concini, Kac and Procesi (see Section 3.3). 
The algebra $U_q[w]$ supports a rational torus action and the theory of Cauchon and Goodearl-Letzter can be applied to this algebra. The torus-invariant primes of this algebra have been studied recently and independently by Cauchon and M\'eriaux on one hand \cite{CM} and Yakimov on the other hand \cite{Y}.  As a consequence of our formula, we are able to give a formula for the dimension of the $(0)$-stratum of $U_q[w]$ which only depends on the Weyl group element $w$.

Regarding the $\hc$-strata in $\oqmmn$, we show even more.  We say that an $m\times n$ Cauchon diagram $C$ contains another $m\times n$ Cauchon diagram $C'$ if whenever a square is coloured black in $C'$, the corresponding square is also coloured black in $C$.  After additional investigations we are able to prove the following result.
\begin{thm} \label{intro2}
Let $P$ be an $\hc$-prime of $\oqmmn$ whose associated $\hc$-stratum is $d$-dimensional.  Then there exists a chain $$P=P_0\subsetneq P_1 \subsetneq \cdots \subsetneq P_d$$ of $\hc$-primes such that the dimension of the $\hc$-stratum associated to $P_i$ is $d-i$ and such that $$C_0\subsetneq C_1 \subsetneq \cdots \subsetneq C_d,$$
where $C_i$ is the Cauchon diagram associated to $P_i$.
\end{thm}

To prove this chain result, we need to understand the relation between inclusion of Cauchon diagrams and inclusion of the corresponding $H$-primes.  One might naively expect these two posets to be isomorphic, but this is not the case.  For instance, 
consider the algebra of $2 \times 2$ quantum matrices $\oqmtwo$, which is generated by four indeterminates $Y_{1,1}, Y_{1,2},Y_{2,1},Y_{2,2}$ subject to the relations given in the beginning of this section.  It is well known that the ideal $ ( Y_{1,1} Y_{2,2}-q Y_{1,2}Y_{2,1} )$ generated by the quantum determinant and $( Y_{2,1},Y_{2,2} )$ are $\hc$-invariant prime ideals in $\oqmtwo$. Clearly, 
$ ( Y_{1,1} Y_{2,2}-q Y_{1,2}Y_{2,1} ) \subsetneq ( Y_{2,1},Y_{2,2} )$, but the corresponding Cauchon diagrams, which can be represented by the pictures in Figure \ref{intro-figure2}, are not comparable.

\begin{figure}
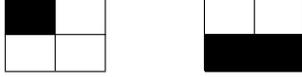
 

\begin{tabular}{|p{0.30cm}|p{0.30cm}|p{0.30cm} p{0.30cm}|p{0.30cm}|p{0.30cm}|}

\cline{1-2} \cline{5-6}
\gray &  &  &  &  & 
 \\
\cline{1-2} \cline{5-6}
 & &  &  & \gray& \gray
 \\
\cline{1-2} \cline{5-6}
\end{tabular}
\caption{Cauchon diagrams representing the ideal generated by the quantum determinant and by $Y_{2,1},Y_{2,2}$ respectively in $\oqmtwo$.}
\label{intro-figure2}
\end{figure}

Thus two $\hc$-primes can be comparable (for the inclusion) and yet their corresponding Cauchon diagrams may fail to be comparable.  Interestingly, if we consider things from the other direction, we see there is a definite relation between these two posets.

\begin{thm}
\label{intro-increasing1}
If $C$ and $C'$ are two $m\times n$ Cauchon diagrams with $C \subsetneq C'$, then $J_{C} \subsetneq J_{C'}$, where $J_{C}$ and $J_{C'}$ denote respectively the $\hc$-primes of $\oqmmn$ associated to $C$ and $C'$. 
\end{thm}
Again, we are able to prove Theorem \ref{intro-increasing} for a much broader class of algebras, called CGL extensions. This class of algebras contains quantum affine spaces, the algebra of $m \times n$ quantum matrices, positive parts of quantized enveloping algebras, and many other interesting families of algebras.  These algebras support a rational action by a torus $H$ and have the property that they have only finitely many $H$-strata in the Goodearl-Letzter stratification. 

The deleting derivations theory of Cauchon applies to CGL extensions and this gives new insights 
into the $H$-stratification in these cases.  The $H$-primes are, just as in the case of quantum matrices, in $1$-$1$ correspondence with combinatorial objects called Cauchon diagrams. These diagrams depend on the algebra and a CGL extension other than quantum matrices has a different collection of Cauchon diagrams than the ones described earlier for $m\times n$ quantum matrices.  Cauchon's \cite{c2} original description was just for quantum matrices, but other authors have since applied his deleting derivations theory to other classes of algebras. Lenagan, Rigal and the second named author \cite{llr2} gave a description of Cauchon diagrams for the quantum Grassmannian. M\'eriaux \cite{meriaux} gave a description of these diagrams for the positive part of the quantized enveloping algebra of a simple Lie algebra, while Cauchon and M\'eriaux have recently described Cauchon diagrams in quantum Schubert cells \cite{CM}. The set of Cauchon diagrams has a natural poset structure under inclusion; likewise the set of $H$-primes can be viewed as a poset under inclusion. 

In this broader context, we are able to show in fact that the following result holds.
\begin{thm}
\label{intro-increasing}
If $w$ and $w'$ are two Cauchon diagrams of a CGL extension $R$ with $w \subsetneq w'$, then $J_{w} \subsetneq J_{w'}$, where $J_{w}$ and $J_{w'}$ denote respectively the (unique) $H$-primes associated to $w$ and $w'$. 
\end{thm}
Note that Theorem \ref{intro-increasing1} is a special case of this result.

The outline of this paper is as follows.  In \S 2 we give the necessary background on CGL extensions and we prove Theorem \ref{intro-increasing}.  In \S 3, we give a formula for the dimension 
of a stratum in a uniparameter CGL extension. Then we use this formula 
to compute the dimension of the $(0)$-stratum in a quantum Schubert cell. In \S 4 we describe the results obtained in the previous sections in the particular case of quantum matrices. Then we use these results in order to prove Theorems \ref{intro-bound} and \ref{intro2}.  In \S 5, after having seen the possible values that can occur as the dimension of an $\hc$-stratum in quantum matrices, we give a conjecture about the number of $d$-dimensional $\hc$-strata in $\oqmmn$.

Throughout this paper, we use the following conventions. 
\begin{enumerate}
\item[$(i)$]
 If
$I$ is a finite set, $|I|$ denotes its cardinality.  
\item[$(ii)$]  $\gc a,b \dc := \{ i\in{\mathbb N} \mid a\leq i\leq b\}$. 
\item[$(iii)$] $\mathbb{K}$ denotes a field and we set
$\mathbb{K}^*:=\mathbb{K}\setminus \{0\}$.  
\item[$(iv)$] If $A$ is a $\mathbb{K}$-algebra, then $\spec(A)$ and  $\prim(A)$ denote respectively its prime and primitive spectra.
\end{enumerate}

\section{$H$-primes in CGL extensions}
\label{CGL}

In this section, we recall the notion of CGL extensions that was introduced in \cite{llr}. 
Examples include various quantum algebras in the generic case such as quantum affine spaces, quantum matrices, 
the positive part of quantized enveloping algebras of semisimple complex Lie algebras, etc. As we will see, the advantage of this class of algebras is that one can use both the stratification theory of Goodearl and Letzter and the theory of deleting derivations of Cauchon in order to study their prime and primitive spectra. This will allow us to investigate the topology of the $H$-spectrum of such algebras and prove Theorem \ref{intro-increasing}. 

\subsection{$H$-stratification theory of Goodearl and Letzter, and CGL extensions}
\label{sectionCGL}
 
\vskip 2mm

Throughout this subsection, $N$ denotes a positive integer and
 $R$ is an iterated Ore extension; that is,
\begin{equation}
R\ = \ \C[X_1][X_2;\sigma_2,\delta_2]\cdots[X_N;\sigma_N,\delta_N],
\end{equation}
 where $\sigma_j$ is an automorphism of the $\C$-algebra $R_{j-1}:=\C[X_1][X_2;\sigma_2,\delta_2]\dots[X_{j-1};\sigma_{j-1},\delta_{j-1}]$
 and  $\delta_j$ is a  $\C$-linear $\sigma_j$-derivation of
 $R_{j-1}$ for all $j\in \gc 2 ,N \dc$. In other words, $R$ is a skew polynomial ring whose multiplication is defined by:
 $$X_j a = \sigma_j(a) X_j + \delta_j(a)$$
 for all $j\in \gc 2 ,N \dc$ and $a \in R_{j-1}$. Thus $R$ is a noetherian domain.  Henceforth, we
 assume that, in the terminology of \cite{llr}, $R$ is a CGL extension.

\begin{defn}
\label{defCGL}
{\em 
The iterated Ore extension $R$ is said to be a \emph{CGL extension} if 
\begin{enumerate}
\label{hypofond}
\item For all $j \in \gc 2,N \dc$, $\delta_j$ is locally nilpotent;
\item For all $j \in \gc 2,N \dc$, there exists  $q_j \in \C^*$ such
  that $\sigma_j \circ \delta_j = q_j \delta_j \circ \sigma_j$ 
and, for all $i \in \gc 1,j-1 \dc$, there exists $\lambda_{j,i} \in
  \C^*$ such that $\sigma_j(X_i)=\lambda_{j,i} X_i$;
\item None of the $q_j$ ($2 \leq j \leq N$) is a root of unity;
\item There exists a torus $H=(\C^*)^d$ that acts rationally by
  $\C$-automorphisms on $R$ such that:
  \begin{itemize}
\item $X_1,\dots,X_N$ are $H$-eigenvectors;
\item The set $\{ \lambda \in \C^* \mbox{ $\mid$ } (\exists h
  \in H)(h\cdot X_1=\lambda X_1)\}$ is infinite;
\item For all $j \in \gc 2,N \dc$, there exists $h_j \in H$ such
  that $h_j\cdot X_i=\lambda_{j,i}X_i$ if $1 \leq i < j$ 
and $h_j\cdot X_j=q_j X_j$.
\end{itemize}
\end{enumerate}
}
\end{defn}

It follows from work of Goodearl and Letzter \cite{gl2} that every $H$-prime ideal of $R$ is completely prime, so $H$-$\spec(R)$ coincides with the set of $H$-invariant completely prime ideals of $R$. Moreover there are at most $2^N$ $H$-prime ideals in $R$. As a corollary, the $H$-stratification (see (\ref{eq:Hstratification})) breaks down the prime spectrum
of $R$ into a finite number of parts, the $H$-strata. The geometric
nature of the $H$-strata is well known: each $H$-stratum is
homeomorphic to the scheme of irreducible subvarieties of a $\C$-torus \cite[Theorems II.2.13 and II.6.4]{bgbook}. However the dimension of these schemes are unknown in general.

\subsection{Quantum affine spaces}
\vskip 2mm

We now recall an important subclass of CGL extensions, namely quantum affine spaces. 

Let $N$ be a positive integer and let $\Lambda=\left( \Lambda_{i,j}
\right) \in M_N (\C^*)$ be a multiplicatively antisymmetric
matrix; that is, $\Lambda_{i,j} \Lambda_{j,i}=\Lambda_{i,i}=1$
 for all  $i,j \in \gc 1,N \dc$. The quantum affine space associated
 to $\Lambda$ is denoted by  $\mathcal{O}_{\Lambda}(\C^N)=
\C_{\Lambda}[T_1,\dots,T_N]$; this is the $\C$-algebra generated by
  $N$ indeterminates $T_1,\dots,T_N$ subject to the relations $T_j T_i
  = \Lambda_{j,i} T_i T_j$ for all $i,j \in \gc 1,N \dc$. It is
  well known that $\mathcal{O}_{\Lambda}(\C^N)$ is an iterated Ore
  extension that we can write:
$$\mathcal{O}_{\Lambda}(\C^N)=\C [T_1][T_2;\sigma_2]\cdots[T_N;\sigma_N],$$
where $\sigma_j$ is the automorphism defined by
$\sigma_j(T_i)=\Lambda_{j,i} T_i$ for all $1\leq i < j \leq
N$. Observe that the torus $H=\left( \C ^* \right) ^N$ acts by
automorphisms on $\mathcal{O}_{\Lambda}(\C^N)$ 
via:
$$(a_1,\dots,a_N)\cdot T_i=a_iT_i \mbox{ for all }i\in \gc 1,N \dc \mbox{ and }(a_1,\dots,a_N) \in H.$$
 Moreover, it is well known (see for instance \cite[Corollary 3.8]{llr}) that
 $\mathcal{O}_{\Lambda}(\C^N)$ is a CGL extension 
 with this action of $H$. Hence $\mathcal{O}_{\Lambda}(\C^N)$ has at
 most $2^N$  $H$-prime ideals and they are all completely prime.

The $H$-stratification of 
$\spec\left( \mathcal{O}_{\Lambda}(\C^N) \right)$ has been entirely
described by Brown and Goodearl when the group $\langle \Lambda_{i,j}
\rangle$ is torsion free  \cite{bgtrans} and next by Goodearl and Letzter
in the general case \cite{gl1}. We now recall their results.

Let $W$ denote the set of subsets of $\gc 1,N \dc$. If $w \in W$,
then we denote by $K_w$ the (two-sided) ideal of
$\mathcal{O}_{\Lambda}(\C^N)$ generated by the indeterminates $T_i$
with $i\in w$. It is easy to check that $K_w$ is an $H$-invariant completely prime ideal of 
$\mathcal{O}_{\Lambda}(\C^N)$.  

\begin{prop}\cite[Proposition 2.11]{gl1}
\label{proprappelHstratificationespacesaffinesquantiques}
\begin{enumerate} 
\item The ideals $K_w$ with $w\in W$ are exactly the $H$-prime ideals
  of $\mathcal{O}_{\Lambda}(\C^N)$.  Hence there are exactly $2^N$
  $H$-prime ideals in this case;
\item For all  $w\in W$, the
  $H$-stratum associated to $K_w$ is given by
$$\spec_{K_w}\left( \mathcal{O}_{\Lambda}(\C^N) \right) =\left\{ P \in \spec \left(\mathcal{O}_{\Lambda}(\C^N) \right) ~|~ P \cap \{T_i | i\in \gc 1,N \dc \}= \{ T_i ~|~ i\in w\}\right\} .$$
\end{enumerate}
\end{prop}

\subsection{The canonical  partition of ${\rm Spec}(R)$}
\label{canonicalembedding}
\vskip 2mm

In this subsection, $R$ denotes a CGL extension as in Section \ref{sectionCGL}. We present the canonical partition of ${\rm Spec}(R)$
that was constructed by Cauchon \cite{c1}. This partition gives new
insights into the $H$-stratification of ${\rm Spec}(R)$.

In order to describe the prime spectrum of $R$, Cauchon \cite[Section 3.2]{c1} has
constructed an algorithm called the \emph{deleting derivations algorithm}. This algorithm constructs, for each $j\in \gc N+1, 2 \dc$, a family 
$\{X_1^{(j)},\dots,X_N^{(j)}\}$ of elements of the division ring of fractions ${\rm Fract}(R)$ of $R$ defined as follows:

\begin{enumerate}
\item When $j=N+1$, we set $(X_1^{(N+1)},\dots,X_N^{(N+1)})=(X_1,\dots,X_N)$.
\item Assume that $j<N+1$ and that the $X_i^{(j+1)}$ ($i \in \gc 1,N \dc$) are already constructed. 
Then it follows from \cite[Th\'eor\`eme 3.2.1]{c1} that $X_j^{(j+1)} \neq 0$ and that, for each $i\in \gc 1,N \dc$, we have
$$X_i^{(j)}=\left\{ 
\begin{array}{ll}
X_i^{(j+1)} & \quad \mbox{ if }i \geq j \\
\displaystyle{\sum_{k=0}^{+\infty } \frac{(1-q_j)^{-k}}{[k]!_{q_j}} }
\delta_j^k \circ \sigma_j^{-k} (X_i^{(j+1)}) (X_j^{(j+1)})^{-k} & \quad
\mbox{ if }i < j,
\end{array} \right. $$
\end{enumerate}
where $[k]!_{q_j}=[0]_{q_j} \times \dots \times [k]_{q_j}$ with 
$[0]_{q_j}=1$ and $[i]_{q_j}=1+q_j+\dots+q_j^{i-1}$ when $i \geq 1$.

For all $j \in \gc 2,N+1 \dc$, we denote by $R^{(j)}$ the subalgebra of ${\rm Fract}(R)$ generated by the $X_i^{(j)}$; that is,
$$R^{(j)}:= \mathbb{K} \langle X_1^{(j)},\dots,X_N^{(j)} \rangle .$$

The following results were proved by Cauchon \cite[Th\'eor\`eme 3.2.1 and Lemme 4.2.1]{c1}. 

For $j \in \gc 2,N+1 \dc$, we have

\begin{enumerate}
\item $R^{(j)}$ is isomorphic to an iterated Ore extension of the form $$\C [Y_1]\dots[Y_{j-1};\sigma_{j-1},\delta_{j-1}][Y_j;\tau_j]\cdots[Y_N;\tau_N]$$ 
by an isomorphism that sends $X_i^{(j)}$ to $Y_i$ ($1 \leq i \leq N$), where 
$\tau_j,\dots,\tau_N$ denote the $\mathbb{K}$-linear automorphisms  such that 
$\tau_{\ell}(Y_i)=\lambda_{\ell,i} Y_i$ ($1 \leq i \leq \ell$).
\item Assume that $j \neq N+1$ and set $S_j:=\{(X_j^{(j+1)})^n \mbox{| } n\in \mathbb{N} \}=
\{(X_j^{(j)})^n \mbox{ | } n\in \mathbb{N} \}$.
\\This is a multiplicative system of regular elements of $R^{(j)}$ and $R^{(j+1)}$, that satisfies the Ore condition
in $R^{(j)}$ and $R^{(j+1)}$. Moreover we have 
$$R^{(j)}S_j^{-1}=R^{(j+1)}S_j^{-1}.$$
\end{enumerate}

It follows from these results that, for all $j\in \gc 2,N+1 \dc$, $R^{(j)}$ is a noetherian domain.

As in \cite{c1}, we use the following notation.

\begin{notn}
We set $\overline{R}:=R^{(2)}$ and $T_i:=X_i^{(2)}$ for all $i\in \gc 1,N \dc$.
\end{notn}

It follows from \cite[Proposition 3.2.1]{c1} that  $\overline{R}$ is a quantum affine space in the indeterminates $T_1,\dots,T_N$--it is for this reason that Cauchon used the expression ``effacement des d\'erivations". More precisely, let $\Lambda=\left( \mu_{i,j} \right) \in M_N(\C^*)$ be the
multiplicatively antisymmetric matrix whose entries are defined as follows. 
$$\mu_{j,i}=\left\{ \begin{array}{ll}
\lambda_{j,i} & \mbox{ if }i<j \\
1 & \mbox{ if } i=j \\
\lambda_{i,j}^{-1} & \mbox{ if }i> j,
\end{array}\right.$$
where the  $\lambda_{j,i}$ with $i<j$ come from the CGL extension structure of $R$ (see Definition \ref{defCGL}). Then we have \begin{equation}\overline{R}=\mathbb{K}_\Lambda [T_1,\dots,T_N]=\mathcal{O}_{\Lambda}(\C^N). \end{equation}

The deleting derivations algorithm was used by Cauchon in order to relate the prime spectrum of a CGL extension $R$ to the prime spectrum of the associated quantum affine space $\overline{R}$. More precisely, Cauchon  has used this algorithm to construct embeddings \begin{equation}
\varphi_j:{\rm Spec}(R^{(j+1)}) \longrightarrow {\rm Spec}(R^{(j)}) \qquad {\rm for ~} j \in \gc 2,N \dc.
\end{equation} Recall from \cite[Section 4.3]{c1} that these embeddings are defined as follows.

Let $P \in  {\rm Spec}(R^{(j+1)})$. Then 
$$ \varphi_j (P) = \left\{\begin{array}{ll} 
PS_j^{-1} \cap R^{(j)} & \mbox{ if } X_j^{(j+1)} \notin P \\
g_j^{-1} \left( P/(X_j^{(j+1)}) \right) & \mbox{ if } X_j^{(j+1)} \in P \\
\end{array}\right.$$
where $g_j$ denotes the surjective homomorphism 
$$g_j:R^{(j)}\rightarrow R^{(j+1)}/(X_j^{(j+1)})$$ defined by $$g_j(X_i^{(j)}):=X_i^{(j+1)} + (X_j^{(j+1)}).$$ (For more details see \cite[Lemme 4.3.2]{c1}.)  It was proved by Cauchon \cite[Proposition 4.3.1]{c1} that $\varphi_j$ induces an increasing homeomorphism from the topological space $$\{P \in  {\rm Spec}(R^{(j+1)}) \mid X_j^{(j+1)} \notin P \}$$ onto $$\{Q \in  {\rm Spec}(R^{(j)}) \mid X_j^{(j)} \notin Q \}$$ whose inverse is also an increasing homeomorphism; also, $\varphi_j$ induces an increasing homeomorphism from $$\{P \in  {\rm Spec}(R^{(j+1)}) \mid X_j^{(j+1)} \in P \}$$ onto its image by $\varphi_j$ whose inverse similarly is an increasing homeomorphism. Note however that, in general, $\varphi_j$ is not an homeomorphism from ${\rm Spec}(R^{(j+1)})$ onto its image.

Composing these embeddings, we get an embedding \begin{equation} 
\varphi:=\varphi_2 \circ \dots \circ \varphi_N : 
{\rm Spec}(R) \longrightarrow {\rm Spec}(\overline{R}), \end{equation} which is called the  \emph{canonical embedding} from ${\rm Spec}(R)$ into 
${\rm Spec}(\overline{R})$. This canonical embedding allows the construction of a partition of ${\rm Spec}(R)$ as follows.
 
We keep the notation of the previous sections. In particular, $W$ still denotes the set of all subsets of $\gc 1,N \dc$. If $w \in W$, then we
set 
$$  \spec_{w}\left(R \right):=\varphi^{-1} \left(  \spec_{K_w}\left( \overline{R} \right) 
\right).$$
Moreover, we denote by $W'$ the set of those $w \in W$ such that $
\spec_{w}\left(R \right) \neq \emptyset$. The elements of $W'$ are called the {\em Cauchon diagrams} of the CGL extension $R$. Then it follows from the work of Cauchon
\cite[Proposition 4.4.1]{c1} that 
$${\rm Spec}(R) =\bigsqcup_{w \in W'} \spec_w (R) \mbox{ and }~| W'|~ \leq~| W |~ = ~2^N.$$

This partition is called the \emph{canonical partition} of ${\rm Spec}(R)$; this
  gives another way to understand the $H$-stratification, as Cauchon
  has shown \cite[Th\'eor\`eme 5.5.2]{c1} that these two partitions
  coincide. As a consequence, he has given another description of
  the $H$-prime ideals of $R$.

\begin{prop}\cite[Lemme 5.5.8 and Th\'eor\`eme 5.5.2]{c1}
\label{descriHprem}
$ $
\begin{enumerate}
\item Let $w \in W'$. There exists a (unique) $H$-invariant (completely) prime ideal $J_w$ of
  $R$ such that $\varphi(J_w)=K_w$, where $K_w$ denotes the ideal of $\overline{R}$ generated by the $T_{i}$  with $i \in w$.
\item $ H \mbox{-}{\rm Spec}(R) = \{J_w \mbox{ $\mid$ }w\in W'\}$. 
\item $\spec_{J_w}(R)=\spec_w(R)$ for all $w \in W'$.
\end{enumerate}
\end{prop}

\subsection{The map $w \mapsto J_w$ is increasing}

In this section, we prove Theorem \ref{intro-increasing}.

\begin{thm} \label{increasing} Let $R$ be a CGL extension and let $w,w' \in W'$ be two Cauchon diagrams of $R$. If $w \subseteq w'$ then $J_w \subseteq J_{w'}$. 
\end{thm}
\begin{proof}
For $j \in \gc 2,  N+1\dc$ and $P \in {\rm Spec}(R)$, we set $P^{(j)}:= \varphi_j \circ \dots \circ \varphi_N (P)$. \\ 
We prove by induction on $j$ that $$J_w^{(j)} \subseteq J_{w'}^{(j)}\qquad {\rm for~} j \in \gc 2,  N+1\dc .$$

When $j=2$, we have $$J_w^{(2)}=K_w ~~~{\rm and}~~~ J_{w'}^{(2)}=K_{w'}.$$ As $w \subseteq w'$, we have $K_w \subseteq K_{w'}$, so 
$$J_w^{(2)} \subseteq J_{w'}^{(2)},$$ as desired.

We assume $j \leq N$ and $J_w^{
(j)} \subseteq J_{w'}^{(j)}$. We need to prove that $J_w^{(j+1)} \subseteq J_{w'}^{(j+1)}$. 
Observe first that if $$X_j^{(j)} \notin  J_{w'}^{(j)},$$ then $$J_w^{(j)},~J_{w'}^{(j)} \in \{Q \in  {\rm Spec}(R^{(j)}) \mid X_j^{(j)} \notin Q \}.$$ As $\varphi_j$ induces an increasing homeomorphism, still denoted $\varphi_j$, from $$\{P \in  {\rm Spec}(R^{(j+1)}) \mid X_j^{(j+1)} \notin P \}$$ onto $$\{Q \in  {\rm Spec}(R^{(j)}) \mid X_j^{(j)} \notin Q \}$$ whose inverse is also an increasing homeomorphism, we obtain
$$J_w^{(j+1)}=\varphi_j^{-1}(J_w^{(j)}) \subseteq \varphi_j^{-1}(J_{w'}^{(j)})=J_{w'}^{(j+1)},$$ as desired. 

Similarly, if $X_j^{(j)} \in  J_{w}^{(j)}$, then both $J_w^{(j)}$ and $J_{w'}^{(j)}$ belong to $$\varphi_j \left( \{P \in  {\rm Spec}(R^{(j+1)}) \mid X_j^{(j+1)} \in P \} \right).$$ As $\varphi_j$ induces an increasing homeomorphism, still denoted $\varphi_j$, from $$\{P \in  {\rm Spec}(R^{(j+1)}) \mid X_j^{(j+1)} \in P \}$$ onto its image whose inverse is also an increasing homeomorphism, we get 
$$J_w^{(j+1)}=\varphi_j^{-1}(J_w^{(j)}) \subseteq \varphi_j^{-1}(J_{w'}^{(j)})=J_{w'}^{(j+1)},$$ as desired. 

It only remains to deal with the case where $X_j^{(j)} \in  J_{w'}^{(j)}$ and $X_j^{(j)} \notin  J_{w}^{(j)}$, so 
$$J_{w}^{(j)}=J_{w}^{(j+1)}S_j^{-1} \cap R^{(j)} , \ J_{w}^{(j+1)}=J_{w}^{(j)}S_j^{-1} \cap R^{(j+1)} \ \mbox{and} \ J_{w'}^{(j)}=g_j^{-1} \left( J_{w'}^{(j+1)}/(X_j^{(j+1)}) \right) .$$

In order to simplify the notation we set:
\begin{equation}
Y_i:=X_i^{(j)} \mbox{ and } Z_i:=X_i^{(j+1)}.
\end{equation}
Also we let $A$ denote the subalgebra of $R^{(j+1)}$ generated by the $Z_i$ with $i \neq j$;
that, is
\begin{equation} A := \mathbb{K}\langle Z_1,\ldots ,Z_{j-1},Z_{j+1},\ldots ,Z_N\rangle .
\end{equation} 
Observe that $R^{(j+1)}=A[Z_j; \sigma , \delta]$ where $\sigma$ denotes the automorphism of $A$ defined by $\sigma(Z_i) =\lambda_{j,i}Z_i=\sigma_j(Z_i)$ if $i < j$ and $\sigma(Z_i) =\lambda_{i,j}^{-1}Z_i$ otherwise;
and where $\delta$ denotes the $\sigma$-derivation of $A$ defined by $\delta (Z_i)= \delta_j (Z_i)$ if $i <j$ and $\delta (Z_i)= 0$ otherwise. One can easily check that this Ore extension satisfies the conditions 
of \cite[Section 2]{c1}, so that the map $\theta: A \rightarrow R^{(j+1)}S_j^{-1}=R^{(j)}S_j^{-1}$ defined by 
$$\theta (a) = \sum_{k=0}^{+\infty } \frac{(1-q_j)^{-k}}{[k]!_{q_j}} \delta^k \circ \sigma^{-k} (a) Z_j^{-k} \ \ {\rm for ~}a \in A$$
is an homomorphism. Observe that by the definition of the deleting derivations algorithm we have $\theta (Z_i)=Y_i$ for all $i \neq j$. 
Hence $\theta (A) \subset R^{(j)}$ and $R^{(j)}$ is the subalgebra of 
$R^{(j+1)}S_j^{-1}$ generated by $\theta (A)$ and $Z_j=Y_j$.   

In order to simplify the notation, we set \begin{equation}
P:=J_{w}^{(j+1)},~~ P':=J_{w}^{(j)},~~ Q =J_{w'}^{(j+1)},~~{\rm and}~~ Q':=J_{w'}^{(j)}.
\end{equation} 
Let $z \in P$ with $z \neq 0$. We need to prove that $z \in Q$. First we can write
$$z= \sum_{t=0}^d a_t Z_j^t,$$
where $a_t \in A$ and $a_d \neq 0$; moreover, this expression for $z$ in this form is unique.
If $a_0 =0$, then $z \in (Z_j) \subseteq Q$. So we assume that $a_0 \neq 0$. 

For every $t$, there exists $k_t$ minimal such that $\delta^{k_t+1}(a_t)=0$ (recall that $\delta$ is locally nilpotent), so 
$$\theta (a_t)=\sum_{k=0}^{k_t } \frac{(1-q_j)^{-k}}{[k]!_{q_j}} \delta^k \circ \sigma^{-k} (a_t) Z_j^{-k} . $$
By induction on the degree of local nilpotency we get: 
$$a_t=\theta (a_t)+\sum_{k=1}^{k_t } \mu_{k,t} \theta(\delta^k \circ \sigma^{-k} (a_t)) Z_j^{-k}=\sum_{k=0}^{k_t } \mu_{k,t} \theta(\delta^k \circ \sigma^{-k} (a_t)) Z_j^{-k},  $$
where $\mu_{0,t}=1$ and $\mu_{k,t} \in \mathbb{K}$. 
Let $m$ be the maximum of the $k_t-t$. 
Then $$zZ_j^m = \sum_{t=0}^d a_t Z_j^{t+m}=\sum_{t=0}^d \sum_{k=0}^{k_t } \mu_{k,t}\theta(\delta^k \circ \sigma^{-k} (a_t)) Z_j^{t+m-k} \in R^{(j)}.$$ 
Thus $zZ_j^m \in R^{(j)} \cap PS_j^{-1}=P'$. Hence 
\begin{equation}
\label{zz}
zZ_j^m =\sum_{t=0}^d \sum_{k=0}^{k_t } \mu_{k,t} \theta(\delta^k \circ \sigma^{-k} (a_t)) Z_j^{t+m-k} \in P'.
\end{equation}
We let $A'$ denote the subalgebra of $R^{(j)}$ generated by $Y_i$ with $i \neq j$, or equivalently, the image of $\theta$. 
As $P'$ is an $H$-prime ideal of $R^{(j)}=A'[Z_j;\sigma]$ it follows from \cite[Corollary 2.4]{llr} that the coefficient of $Z_j^{\ell}$ in the previous sum belongs to $P'$ for every nonnegative integer $\ell$. In particular, the coefficient of degree $m$ is in $P'$. Hence, by setting $k=t$ in equation (\ref{zz}), we obtain
$$\sum_{t=0}^d \mu_{t,t} \theta(\delta^t \circ \sigma^{-t} (a_t)) \in P'\subseteq Q'.$$
As $Q'=g_j^{-1} (Q/(Z_j))$ and $(Z_j) \subseteq Q$, we get that  
\begin{eqnarray}
\label{eq}
\sum_{t=0}^d \mu_{t,t} \delta^t \circ \sigma^{-t} (a_t) \in  Q.
\end{eqnarray}
As $Z_j \in Q$, we see $\delta(a) = Z_ja - \sigma(a) Z_j \in Q$ for every $a \in A$. Hence we deduce from (\ref{eq}) that 
$a_0=\mu_{0,0}a_0 \in Q$. As $(Z_j) \subseteq Q$, we see that 
$$z= a_0+ \left( \sum_{t=1}^d a_t Z_j^{t-1} \right)Z_j \in Q,$$ as desired.
\end{proof}


\section{Dimension of $H$-strata of uniparameter CGL extensions}

In this section, we obtain a formula for the dimension of a stratum of a uniparameter CGL extension, and apply it to compute the dimension of the $(0)$-stratum of a quantum Schubert cell.\\ 

\subsection{Uniparameter CGL extensions}

In this section, we assume that $R$ is a {\it uniparameter CGL extension}, that is, $R$ is a CGL extension such that there exist an antisymmetric matrix $(a_{i,j}) \in \mathcal{M}_N(\mathbb{Z})$ and $q \in \C^*$ not a root of unity  such that $\lambda_{j,i}=q^{a_{j,i}}$ for all $1 \leq i < j \leq N$.

\subsection{Dimension of $H$-strata of uniparameter CGL extensions}

The aim of this section is to give a formula for the dimension of the $H$-stratum in $R$ associated to a Cauchon diagram $w \in W'$. We need to introduce the following definition.

\begin{defn} {\em
Let $w \in W'$ be a Cauchon diagram of $R$. Let $\{\ell_1< \cdots < \ell_d\}:=\gc 1,N \dc \setminus w$ be the complement of $w$. We define the \emph{skew-adjacency matrix}, $M_R(w)$, of $w$ to be the $d\times d$ matrix whose $(i,j)$ entry is $a_{\ell_i \ell_j}$.
} \end{defn}

\begin{thm}\label{dimstratumGen}
Let $w \in W'$. The $H$-stratum associated to $J_w$ is homeomorphic to the prime spectrum of a commutative Laurent polynomial ring over $\mathbb{K}$ in 
$\dim_{\mathbb{Q}} (\ker (M_R(w)))$ indeterminates. 
\end{thm}
\begin{proof} Let $w \in W'$ be a Cauchon diagram of $R$. Let $\{\ell_1< \cdots < \ell_d\}:=\gc 1,N \dc \setminus w$ be the complement of $w$.  Recall from the work of Cauchon \cite[Th\'eor\`emes 5.1.1 and 5.5.1]{c1} that the canonical
embedding induces an inclusion-preserving homeomorphism from the $H$-stratum $ \spec_{J_w}(R)$ of $R$ associated to $J_w$ onto the $H$-stratum $\spec_{K_w}\left( \overline{R} \right)$ of $\overline{R}$ associated to $K_w$. Hence we deduce from Proposition \ref{proprappelHstratificationespacesaffinesquantiques} that 
\begin{eqnarray}
\label{eqstratum}
\spec_{J_w}(R) & \simeq & \spec_{K_w}\left( \overline{R} \right) \nonumber \\
& \simeq & \{ P \in \spec(\overline{R}) ~|~ P \cap \{ T_1, \dots , T_N \} = \{T_i ~|~ i \in w\}  \}.
\end{eqnarray}

Recall that  $\overline{R}=\mathbb{K}_{\Lambda}[T_{1},T_{2},\dots,T_{N}]$,
where $\Lambda$ denotes the $N \times N$ matrix whose entries are defined by $\Lambda_{k,l}
=q^{a_{k,l}}$ for all $k,l \in \gc 1, N \dc$. Let $\Lambda_w$ denote the multiplicatively antisymmetric $d \times d$ matrix whose entries are defined by $(\Lambda_w)_{i,j}=q^{M_R(w)_{i,j}}=q^{a_{\ell_i \ell_j}}$. 

As $K_w$ is the prime ideal generated by the indeterminates $T_{i}$ such that $i \in w$, the algebra $\overline{R}/K_w$ is isomorphic to 
the quantum affine space  $\mathbb{K}_{\Lambda_w}[t_1, \dots , t_d ]$ by an isomorphism that sends 
$T_{\ell_i} +K_w $ to $t_i$ and $T_k$ to $0$ if $k \neq \ell_i$. 

To finish the proof, we use the same idea as in \cite[Corollary 1.3]{laulen}. 

We denote by $P(\Lambda_w)$ the quantum torus associated to $\mathbb{K}_{\Lambda_w}[t_1, \dots , t_d ]$; that is, 
$$P(\Lambda_w):= \mathbb{K}_{\Lambda_w}[t_1, \dots , t_d ] \Sigma^{-1},$$
where $\Sigma$ denotes the multiplicative system of 
$\mathbb{K}_{\Lambda_w}[t_1, \dots , t_d ]$ generated by the normal elements $t_1,\ldots ,t_d$.

It follows from (\ref{eqstratum}) that
\begin{eqnarray*}
\spec_{K_w}\left( \overline{R} \right) &\simeq & \spec_{(0)}\left( \mathbb{K}_{\Lambda_w}[t_1, \dots , t_d ] \right) \\
& \simeq & \spec (P(\Lambda_w)).
\end{eqnarray*}
  Next, $\spec(P(\Lambda_w))$ is Zariski-homeomorphic via extension and contraction to the prime spectrum of the centre $Z(P(\Lambda_w))$ of $P(\Lambda_w)$, by
\cite[Corollary 1.5]{gl1}. Further, as we shall see,
$Z(P(\Lambda_w))$ is a Laurent polynomial ring. To make 
this result precise, we need to introduce the following notation.

If $\underline{s}=(s_1,\dots,s_d) \in \mathbb{Z}^{d}$, then
we set $t^{\underline{s}}:= t_1^{s_1}\dots t_d^{s_d} \in P( \Lambda_w)$.
As in \cite{gl1}, we denote by $\sigma : \mathbb{Z}^{d} \times
\mathbb{Z}^{d} \rightarrow \mathbb{K}^*$ the antisymmetric bicharacter
defined by 
$$\sigma(\underline{s},\underline{t}):=\prod_{i,j=1}^d
(\Lambda_w)_{i,j}^{s_it_j}=q^{\sum_{i,j=1}^d
a_{\ell_i,\ell_j} s_i t_j } \: \quad {\rm for~all}\quad  
\underline{s},\underline{t} \in
\mathbb{Z}^{d}.$$ 
Then it follows from \cite[1.3]{gl1} that the centre
$Z(P(\Lambda_w))$ of $P(\Lambda_w)$ is a Laurent polynomial ring over $\mathbb{K}$ in the variables
$(t^{\underline{b_1}})^{\pm 1},\dots,(t^{\underline{b_r}})^{\pm 1}$, where
$(\underline{b_1},\dots,\underline{b_r})$ is any basis of $$V:=\{\underline{s}
\in \mathbb{Z}^{d} \mid \sigma(\underline{s},-)\equiv 1\}.$$ Since $q$ is not
a root of unity, easy computations show that $\underline{s} \in V$ if and only
if $M_R(w)^t\underline{s}^t=0$. Hence the centre $Z(P(\Lambda_w))$ of $P(\Lambda_w)$
is a Laurent polynomial ring in $\dim_{\mathbb{Q}} (\ker(M_R(w)^t))=\dim_{\mathbb{Q}} (\ker(M_R(w)))$ indeterminates. 

To summarize, we have 
$$ \spec_{J_w}(R)  \simeq  \spec_{K_w}\left( \overline{R} \right) 
\simeq \spec \left( P(\Lambda_w) \right) \simeq \spec \left( Z(P(\Lambda_w)) \right),$$
and $ Z(P(\Lambda_w)) $ is a Laurent polynomial ring in $\dim_{\mathbb{Q}} (\ker(M_R(w)))$ indeterminates, as desired. 
\end{proof}

\subsection{Application to quantum Schubert cells.}

To finish this section, we use Theorem \ref{dimstratumGen} in order to compute the dimension of the $(0)$-stratum of quantum Schubert cells. 

Let us first recall the definition of quantum Schubert cells.

Let $\g$ be a simple Lie $\comp$-algebra of rank $n$. We denote by
$\pi=\{\alpha_1,\dots,\alpha_n\}$ the set of simple roots
associated to a triangular decomposition $\g=\mathfrak{n}^- \oplus
\mathfrak{h} \oplus \mathfrak{n}^+$. Recall that $\pi$ is a basis of a 
Euclidean vector space $E$ over $\mathbb{R}$, whose inner product is
denoted by $(\mbox{ },\mbox{ })$ ($E$ is usually denoted by $
\mathfrak{h}_{\mathbb{R}}^*$ in Bourbaki). We denote by  $W$ the
Weyl group of $\g$, that is, the subgroup of the orthogonal group of
$E$ generated by the reflections $s_i:=s_{\alpha_i}$, for $i \in
\{1,\dots,n\}$, with reflecting hyperplanes $H_i:=\{\beta \in E \mid
(\beta,\alpha_i)=0\}$, $i \in \{1,\dots,n\}$. The length of $w \in W$ is denoted 
by $l(w)$. Further, we denote by $w_0$ the longest element
of $W$. Finally, we denote by $A=(a_{ij}) \in M_n(\mathbb{Z})$ the Cartan
 matrix associated to these data. As $\g$ is simple, $a_{ij} \in \{0,-1,-2,-3\}$ for all $i \neq j$. 

Recall that the scalar product of two roots $(\alpha,\beta)$ is always
an integer. We assume that the short roots have length $\sqrt{2}$.

For all $i \in \{1,\dots,n \}$, set
$q_i:=q^{\frac{(\alpha_i,\alpha_i)}{2}}$ and 
$$\left[ \begin{array}{l} m \\ k \end{array} \right]_i:=
\frac{(q_i-q_i^{-1}) \dots
  (q_i^{m-1}-q_i^{1-m})(q_i^m-q_i^{-m})}{(q_i-q_i^{-1})\dots
  (q_i^k-q_i^{-k})(q_i-q_i^{-1})\dots (q_i^{m-k}-q_i^{k-m})} $$
for all integers $0 \leq  k \leq  m$. By convention, 
$$\left[ \begin{array}{l} m \\ 0 \end{array} \right]_i:=1.$$

The quantised enveloping algebra $U_q(\g)$ of $\g$ over $\comp$ associated to
the previous data is the $\mathbb{K}$-algebra generated by the
indeterminates $E_1,\dots,E_n,F_1,\dots , F_n,K_1^{\pm 1}, \dots, K_n^{\pm 1}$ subject to the following relations:
$$K_i K_j =K_j K_i $$
$$ K_i E_j K_i^{-1}=q_i^{a_{ij}}E_j  \mbox{ and }  K_i F_j K_i^{-1}=q_i^{-a_{ij}}F_j$$
$$E_i F_j -F_jE_i=\delta_{ij} \frac{K_i-K_i^{-1}}{q_i-q_i^{-1}} $$
and the quantum Serre relations:
\begin{eqnarray}
\label{Serrequantique} 
\sum_{k=0}^{1-a_{ij}} (-1)^k  \left[ \begin{array}{c} 1-a_{ij} \\ k
 \end{array} \right]_i E_i^{1-a_{ij} -k} E_j E_i^k=0  \mbox{ } (i \neq  j)
\end{eqnarray}
and 
$$\sum_{k=0}^{1-a_{ij}} (-1)^k  \left[ \begin{array}{c} 1-a_{ij} \\ k
 \end{array} \right]_i F_i^{1-a_{ij} -k} F_j F_i^k=0  \mbox{ } (i \neq  j).$$

We refer the reader to \cite{bgbook,jantzen,josephbook} for more details on
this (Hopf) algebra. Further, as usual, we denote by $U_q^+(\g)$ (resp. $U_q^-(\g)$)  the
subalgebra of $U_q(\g)$ generated by $E_1,\dots,E_n$
(resp. $F_1,\dots,F_n$) and by $U^0$ the subalgebra of $U_q(\g)$ generated by 
$K_1^{\pm 1},\dots, K_n^{\pm 1}$.

To each reduced decomposition of the longest element $w_0$ of the Weyl group $W$ of $\g$, Lusztig has associated a PBW basis of $\U$, see for instance \cite[Chapter 37]{lusztigbook}, \cite[Chapter 8]{jantzen} or \cite[I.6.7]{bgbook}. The construction relates to a braid group action by automorphisms on $\U$. We use the convention of \cite[Chapter 8]{jantzen}. In particular, for any $\alpha \in \pi$, we define the braid automorphism $T_{\alpha}$ of the algebra $U_q(\g)$ as in 
\cite[p. 153]{jantzen}. We set $T_i:=T_{\alpha_i}$. It was proved by Lusztig that the automorphisms $T_{i}$ satisfy the braid relations, that is, if $s_is_j$ has order $m$ in $W$, then $$T_iT_jT_i \dots = T_j T_i T_j \dots ,$$
where there are exactly $m$ factors on each side of this equality.

Consider any $w\in W$, and set $t := l(w)$. Let $w= s_{i_1} \circ \cdots \circ s_{i_t}$ $(i_j \in \{1, \dots , n\})$ be a reduced decomposition of $w$. It is well known that
$\beta_1 = \alpha_{i_1}$, $\beta_2 = s_{i_1}(\alpha_{i_2})$, ..., $\beta_t = s_{i_1} \circ \cdots \circ s_{i_{t-1}}(\alpha_{i_t})$ 
are distinct positive roots and that the set $\{\beta_1, ..., \beta_t\}$ does not depend on the chosen reduced expression of $w$. 
Similarily, we define elements $E_{\beta_k}$ of $U_q(\g)$ by
$$E_{\beta_k}:= T_{i_1} \cdots T_{i_{k-1}} (E_{i_k}).$$
Note that the elements $E_{\beta_k}$ depend on the reduced decomposition of $w$. The following well-known results were proved by Lusztig and Levendorskii-Soibelman.

\begin{thm}[See for instance \cite{LevSoi}]
\label{theofond}
$ $
\begin{enumerate}
 \item For all $k \in \{ 1, \dots, t\}$, the element $E_{\beta_k}$ belongs to $\U$.
 \item If $\beta_k=\alpha_i$, then $E_{\beta_k}=E_i$.
 \item For all $1 \leq i < j \leq t$, we have 
 $$E_{\beta_j} E_{\beta_i} -q^{-(\beta_i , \beta_j)} E_{\beta_i} E_{\beta_j}=
 \sum a_{k_{i+1},\dots,k_{j-1}} E_{\beta_{i+1}}^{k_{i+1}} \cdots E_{\beta_{j-1}}^{k_{j-1}},$$
 where each $a_{k_{i+1},\dots,k_{j-1}} $ belongs to $\C$.
\end{enumerate}
\end{thm}

We denote by $U_q[w]$ the subalgebra of $\U$ generated by $E_{\beta_1}, \dots, E_{\beta_t}$. It is well known that $U_q[w]$ does not depend on 
the reduced decomposition of $w$. Moreover, the monomials $E_{\beta_1}^{k_1} \cdots E_{\beta_t}^{k_t}$, with $k_1, \dots, k_t \in \mathbb{N}$, 
 form a linear basis of $U_q[w]$.
As a consequence of this result, $U_q[w]$ can be presented as a skew-polynomial algebra: 
$$U_q[w]= \C [E_{\beta_1}] [E_{\beta_2}; \sigma_2 , \delta_2] \cdots [E_{\beta_t}; \sigma_t , \delta_t],$$
 where each $\sigma_i$ is a linear automorphism and each $\delta_i$ is a $\sigma_i$-derivation of the appropriate 
 subalgebra. In particular, $U_q[w]$ is a noetherian domain and its group of invertible elements is reduced to nonzero elements of the base-field. 
 
It is well known that the torus $\hc:=(\mathbb{K}^*)^n$ acts rationally by automorphisms on $\U$ via:
$$(h_1,\dots, h_n).E_i = h_i E_i \mbox{ for all } i \in \{1,\dots,n\}.$$
(It is easy to check that the quantum Serre relations are preserved by the group $\hc$.) It is also well known that this action of $\hc$ on $\U$ restricts to a rational action of $\hc$ on $U_q[w]$. Observe that $(0)$ is an $\hc$-prime in $U_q[w]$ as this algebra is a domain.

 It was proved by Cauchon \cite[Proposition 6.1.2 and Lemme 6.2.1]{c1} that $U_q[w]$ is a uniparameter CGL extension with the following associated antisymmetric matrix:
$$\left(
\begin{array}{ccccc}
0                   & (\beta_1 , \beta_2) & \cdots & \cdots & (\beta_1 , \beta_t)\\
-(\beta_1 , \beta_2) & 0  & (\beta_2 , \beta_3)&  & (\beta_2 , \beta_t)\\
\vdots & \ddots & \ddots & \ddots & \vdots \\
\vdots & & \ddots & 0 &(\beta_{t-1} , \beta_t) \\
-(\beta_1 , \beta_t) & \dots & \dots & -(\beta_{t-1} , \beta_t)& 0\\
\end{array}
\right).$$
 The kernel of this matrix has been described by De Concini and Procesi \cite[Lemma 10.4 and 10.6]{DP} who proved that the kernel of this matrix identifies with $\ker (id_E +w)$. So we deduce from Theorem \ref{dimstratumGen} the following result.

\begin{prop}
The dimension of the stratum associated to $(0)$ in $U_q[w]$ is 
$\dim \ker (id_E +w)$. 
\end{prop}

It follows from \cite[Proposition 2.2.1]{CM} that the algebra $\oqmmn$ of quantum matrices can be presented as a quantum Schubert cell. So one can use the previous proposition in order to retrieve the dimension of the $(0)$-stratum of $\oqmmn$ that was first obtained in \cite{laulen}. 

In the next section, we investigate in more details the dimensions of the $\hc$-strata in $\oqmmn$.

\section{Quantum matrices}

In this section we study the dimensions of the strata occurring in the Goodearl-Letzter stratification of the ring of $m\times n$ quantum matrices and we prove Theorems \ref{intro-bound} and \ref{intro2}.   
Throughout this section, $q$ denotes a nonzero element of $\mathbb{K}$ 
that is not a root of unity.  The ring $R=\oqmmn$ of $m\times n$ quantum matrices  was defined in the beginning of \S \ref{int}. The ring $R$ is known to be a quantum Schubert cell by \cite[Proposition 2.2.1]{CM}, so we know (see the previous section) that $\oqmmn$ is a uniparameter CGL extension. Nevertheless 
we start by explicitly describe the results of the previous sections in this situation.  

\subsection{Quantum matrices as a CGL extension}
\label{QMCGL}
\vskip 2mm

This section serves to show that quantum matrix rings give examples of CGL extensions and that we can therefore draw upon the background given in \S \ref{CGL}.

It is well known that $R=\oqmmn$ can be presented as an iterated Ore extension over
$\mathbb{K}$ with the generators $Y_{\ia}$ given in the beginning of \S \ref{int} adjoined in lexicographic order.
Thus the ring $R$ is a noetherian domain; we denote by $F$ its skew-field of
fractions. Moreover, since $q$ is not a root of unity, it follows from
\cite[Theorem 3.2]{gletpams} that all prime ideals of $R$ are completely
prime.

It is well known that the algebras $\oqmmn$ and $\oqmnm$ are
isomorphic. Hence, all the results that we will prove for $\oqmmn$ will also be valid for 
$\co_q(M_{n,m})$. Because of this, we assume that $n \leq m$.

$\hc:=\left( \mathbb{K}^* \right)^{m+n}$ acts on $R$ by
$\mathbb{K}$-algebra automorphisms via:
$$(a_1,\dots,a_m,b_1,\dots,b_n)\cdot Y_{\ia} = a_i b_\alpha Y_{\ia} \quad {\rm
for~all} \quad \: (\ia)\in \gc 1,m \dc \times \gc 1,n \dc.$$ 
Moreover, as $q$ is not a root of unity, $R$ endowed with this action of $\hc$ is a uniparameter CGL extension (see for instance \cite{llr}). 
Before going any further let us describe the antisymmetric matrix associated to the uniparameter CGL extension $\oqmmn$. First, we set
$$A:=\left(
\begin{array}{ccccc}
 0 & 1 & 1 & \dots & 1 \\
-1 & 0 & 1 & \dots  & 1 \\
\vdots & \ddots &\ddots&\ddots &\vdots \\
-1 & \dots & -1 & 0 & 1  \\
 -1& \dots& \dots & -1 & 0 \\  
\end{array} 
\right) 
\in M_{m}(\mathbb{Z}) \subseteq M_{m}(\mathbb{Q}). $$
Then the antisymmetric matrix associated to the uniparameter CGL extension $\oqmmn$ is the matrix $B$ as follows.
 $$B=(b_{k,l}):=\left( 
\begin{array}{ccccc}
 A & I_m & I_m & \dots & I_m \\
-I_m & A & I_m & \dots  & I_m \\
\vdots & \ddots &\ddots&\ddots &\vdots \\
-I_m & \dots & -I_m & A & I_m  \\
 -I_m& \dots& \dots & -I_m & A \\  
\end{array} 
\right)
\in M_{mn}(\mathbb{Q}), $$
where $I_m$ denotes the identity matrix of $M_m(\mathbb{Q})$.

The fact that $R=\oqmmn$ is a (uniparameter) CGL extension implies in particular that $\hc$-$\spec(R)$ is finite and that every $\hc$-prime is completely prime. Also, as $R=\oqmmn$ is a CGL extension, one can apply the results of Section \ref{CGL} to this algebra. In particular, using the theory of deleting derivations, Cauchon has given a combinatorial description 
of $\hc$-$\spec(R)$. More precisely, in the case of the algebra $R=\oqmmn$, he has described the set $W'$ that appeared in Section \ref{canonicalembedding} as follows.

First, it follows from \cite[Section 2.2]{c2} that the quantum affine space $\overline{R}$ that appears in Section \ref{canonicalembedding} is in this case $\overline{R}=\mathbb{K}_{\Lambda}[T_{1,1},T_{1,2},\dots,T_{m,n}]$, 
where $\Lambda$ denotes the $mn \times mn$ matrix whose entries are defined by $\Lambda_{k,l}
=q^{b_{k,l}}$ for all $k,l \in \gc 1, mn \dc$. Using the canonical embedding (see Section \ref{canonicalembedding}), Cauchon \cite{c2} produced a bijection between $\hc$-$\spec(\oqmmn)$ and the collection $\mathcal{C}_{m,n}$ of $m\times n$ Cauchon diagrams as defined in Definition 1. Roughly speaking, with the notation of previous sections, the set $W'$ of Cauchon diagrams coincides with the set of $m \times n$ Cauchon diagrams of Definition 1. Let us make 
this precise. If $C$ is an $m\times n$ Cauchon diagram, then we denote by $K_C$ the (completely) prime ideal 
of $\overline{R}$ generated by the indeterminates $T_{i,\alpha}$ such that the box in position $(i,\alpha)$ is a black box of $C$. Then, with $\varphi : \spec(R) \rightarrow \spec(\overline{R})$ denoting the canonical embedding, it follows from \cite[Corollaire 3.2.1]{c2} that there exists a unique $\hc$-invariant (completely) prime ideal $J_C$ of $R$ such that $\varphi (J_C)=K_C$; moreover there are no other $\hc$-primes in $\oqmmn$; that is,
$$\hc \mbox{-} \spec (\oqmmn)=\{ J_C~ |~ C \in \mathcal{C}_{m,n} \}.$$
This last equality justifies the terminology ``$m\times n$ Cauchon diagrams'' for the combinatorial objects described in Definition 1. 

In light of this, the containment rule for $m \times n$ Cauchon diagrams given in the Introduction coincides exactly with set-theoretic containment for the more general description of Cauchon diagrams in terms of sets. So, in the case of quantum matrices, Theorem \ref{increasing} can be rephrased as follows.

\begin{thm}
\label{notintro-increasing}
If $C$ and $C'$ are two $m\times n$ Cauchon diagrams with $C \subsetneq C'$, then $J_{C} \subsetneq J_{C'}$, where $J_{C}$ and $J_{C'}$ denote respectively the $\hc$-primes of $\oqmmn$ associated to $C$ and $C'$. 
\end{thm}

\subsection{Dimension of $\hc$-strata}
\label{sectionCauchondiagram}
\vskip 2mm
We now give some results about the dimension of the $\hc$-stratum of an $\hc$-prime of $\oqmmn$ corresponding to an $m\times n$ Cauchon diagram $C$.  We give two related definitions.
\begin{defn}
{\em
A Cauchon diagram $C$ is \emph{labelled} if each white box in $C$ is labelled with a positive integer such that:
\begin{enumerate}
\item{the labels are strictly increasing from left to right along rows;}
\item{if $i<j$ then the label of each white box in row $i$ is strictly less than the label of each white box in row $j$.}
\end{enumerate}
}
\end{defn}
\begin{figure}
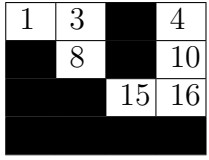

\begin{tabular}{|p{0.30cm}|p{0.30cm}|p{0.30cm}|p{0.30cm}|p{0.30cm}|p{0.30cm}|}

\hline
1 & 3 & \gray & 4 
 \\
\hline
\gray & 8 & \gray & 10 
 \\
\hline
\gray & \gray & 15 & 16  \\
\hline
\gray & \gray & \gray & \gray  \\
\hline
\end{tabular}
\caption{An example of a $4\times 4$ labelled Cauchon diagram.}
\label{lab-figure}
\end{figure}

\begin{defn} {\em
Let $C$ be an $m\times n$ labelled Cauchon diagram with $d$ white boxes and labels $\ell_1< \cdots < \ell_d$. We define the \emph{skew-adjacency matrix}, $M(C)$, of $C$ to be the $d\times d$ matrix whose $(i,j)$ entry is:
\begin{enumerate}
\item 1 if the box labelled $\ell_i$ is strictly to the left and in the same row as the box labelled $\ell_j$ or is strictly above and in the same column as the box labelled $\ell_j$;
\item $-1$ if the box labelled $\ell_i$ is strictly to the right and in the same row as the box labelled $\ell_j$ or is strictly below and in the same column as the box labelled $\ell_j$;
\item 0 otherwise.
\end{enumerate}
} \end{defn}

Observe that $M(C)$ is independent of the set of labels which appear in $C$. 

See, for example, Figure 4.
\begin{figure}[h] \label{fig: 3}
\vskip 2mm
$C:~~$ \begin{tabular}{|p{0.30cm}|p{0.30cm}|p{0.30cm}|p{0.30cm}|}

\hline
$1$ & \gray & \gray &\gray 
 \\
\hline
\gray& \gray & $2$ & \gray 
 \\
\hline
$3$& \gray & $4$ & $5$  
\\
\hline
\gray & \gray & \gray & \gray \\
\hline
\end{tabular}\hskip 3mm $\mapsto$ \hskip 3mm $M(C) = \left( \begin{array}{rrrrr} 0&0&1&0&0\\ 0&0&0&1&0 \\ -1&0&0&1&1 \\ 0&-1&-1&0&1 \\0&0&-1&-1&0\\
\end{array} \right)$
\caption{A labelled Cauchon diagram $C$ and its corresponding skew-adjacency matrix $M(C)$.}
\end{figure}

As a particular case of Theorem \ref{dimstratumGen}, we get the following result for the uniparameter CGL extension $\oqmmn$.

\begin{prop}\label{dimstratum}
Let $C$ be an $m \times n $ Cauchon diagram. The $\hc$-stratum associated to $J_C$ is homeomorphic to the prime spectrum of a commutative Laurent polynomial ring over $\mathbb{K}$ in 
$\dim_{\mathbb{Q}} (\ker (M(C)))$ indeterminates. 
\end{prop}

We now use this result in order to prove the first part of Theorem \ref{intro-bound}. In order to achieve this aim, we need the following lemma.

\begin{lem}\label{lem: 1}
Let $C$ be an $m\times n$ labelled Cauchon diagram with $n\le m$ and with $d$ white boxes with labels $\gc 1,d\dc$. Assume that $C$ has no all black columns.  For $1\le j\le n$, let $a_j$ denote the smallest label which appears in column $j$ of $C$. Then there is a $d\times d$ lower triangular matrix $S$ such that the matrix obtained by deleting columns $a_1,\ldots ,a_n$ and rows $a_1,\ldots ,a_n$ from $S\cdot M(C)$ is invertible.
\end{lem}
\begin{proof} We let $d_i$ denote the number of white boxes in the $i$th row of $C$.  Then
$$d=d_1+\cdots +d_m$$ and we can write $M(C)$ in block form as
\begin{equation}
\label{MC}
 M(C) \ = \ \left(\begin{array}{cccc} A_1 & J_{1,2} & \cdots & J_{1,m} \\
-J_{1,2}^{\rm T} & A_2 & \cdots & J_{2,m} \\
\vdots & \vdots & \ddots & \vdots \\
-J_{1,m}^{\rm T} & -J_{2,m}^{\rm T} & \cdots & A_m \end{array} \right), 
\end{equation}

where $A_i$ is the $d_i \times d_i$ matrix whose diagonal entries are zero, the entries above the diagonal are $1$, and the entries below the diagonal are $-1$; and $J_{i,j}$ is the $d_i\times d_j$ $0,1$-matrix whose $(k,\ell)$ entry is $1$ if the $k$th white element in row $i$ of $C$ (looking from left to right) is directly above the $\ell$th white element in row $j$ of $C$ (again, looking left to right) and is $0$ otherwise.  

We now define
 a $d\times d$ lower-triangular matrix $S$, whose diagonal entries are all $1$ and for $i>j$, the $(i,j)$ entry is $-1$ if the white box labelled $i$ is in the $k$th column of $C$ and $j=a_k$; and the $(i,j)$ entry is $0$ otherwise.
We now consider the product $S\cdot M(C)$.  
\vskip 2mm
\noindent 
{\bf Claim:} If $i<j$ and the boxes labelled $i$ and $j$ are not in the same row of $C$ and $i\not\in \{a_1,\ldots ,a_n\}$ then the $(i,j)$ entry of $S\cdot M(C)$ is zero.
\vskip 2mm 
\noindent {\it Proof of the claim.} Suppose $i<j$ and $i\not\in
\{a_1,\ldots ,a_n\}$.  
Let $\ell$ be the column in which the box labelled $i$ in $C$ sits.  Then
\begin{eqnarray*} (S\cdot M(C))_{i,j} & = & \sum_{k=1}^d S_{i,k}M(C)_{k,j} \\
&=& S_{i,i}M(C)_{i,j} + S_{i,a_{\ell}}M(C)_{i,j} \\
&=& M(C)_{i,j} - M(C)_{a_{\ell},j}. \end{eqnarray*}
By assumption, the boxes labelled $i$ and $j$ are not in the same row of $C$ and hence if $M(C)_{i,j}=1$, then $j$ must also be in the $\ell$th column of $C$.  But then $M(C)_{a_{\ell},j}=1$, and so $(S\cdot M(C))_{i,j}=0$.  Similarly, if $M(C)_{i,j}=0$, then $j$ must be in a different column than $i$ and so $M(C)_{a_{\ell},j}=0$ as well.  The claim follows. \hfill $\square$
\vskip 2mm
Let $D$ denote the matrix obtained from $S\cdot M(C)$ by deleting the rows indexed by $a_1,\ldots ,a_n$ and the columns indexed by $a_1,\ldots ,a_n$.  Let $e_i$ denote the number of labels $\{a_1,\ldots ,a_n\}$ which appear in the $i$th row of $C$.  By the first Claim, $D$ is a block lower-triangular matrix; that is,
\[ D \ = \ \left(\begin{array}{cccc} D_1 & 0 & \cdots & 0 \\
* & D_2 & \cdots & 0  \\
\vdots & \vdots & \ddots & \vdots \\
* & * & \cdots & D_m \end{array} \right), \]
where $D_i$ is a $(d_i-e_i)\times (d_i-e_i)$ matrix.
\vskip 2mm
\noindent 
{\bf Claim:} For $1 \le i\le m$, we have $D_i = -I_{d_i-e_i}+D_i'$ for some skew-symmetric matrix $D_i'$.
\vskip 2mm
\noindent {\it Proof of the claim.} Let us first consider the matrix $S\cdot M(C)$.   Let $\{a+1, a+2,\ldots ,a+d_k\}$ be the set of labels in the $k$th row of $M(C)$.  (Here $a=d_1+d_2+\cdots +d_{k-1}$.)  Consider $(S\cdot M(C))_{a+i,a+j}$.  Since $D$ is obtained by deleting the rows and columns of $S\cdot M(C)$ indexed by $a_1,\ldots ,a_d$, we may assume that $a+i\not \in \{a_1,\ldots, a_n\}$.  In this case, the box labelled $a+i$ appears in the $\ell$th column of $C$ for some $\ell$.  Then $(S\cdot M(C))_{a+i,a+j} = M(C)_{a+i,a+j} - M(C)_{a_{\ell}, a+j}$.  Observe that $M(C)_{a_{\ell},a+j}$ is nonzero
if and only if $a+j$ is also in the $\ell$th column of $C$; but by assumption, $a+i$ is in the $\ell$th column and the boxes labelled $a+i$ and $a+j$ are in the same row, so this is impossible unless $i=j$.  Hence
$$(S\cdot M(C))_{a+i,a+j} = M(C)_{a+i,a+j} - \delta_{i,j},$$ if $a+i\not \in \{a_1,\ldots ,a_n\}$.  
To obtain $D$, we simply delete the rows and columns indexed by $\{a_1,\ldots ,a_n\}$.  In particular, if we let $A_k$ denote  the $d_k\times d_k$ submatrix of $A$ whose rows and columns are indexed by $a+1,\ldots ,a+d_k$ (see (\ref{MC})), then $D_k$ is obtained by deleting the $e_i$ rows and columns indexed by the labels $\{a_1,\ldots ,a_n\}$ which occur in the $i$th row of $C$ and then subtracting the identity.  Since each $A_k$ is skew-symmetric, the claim follows. \hfill $\square$
\vskip 2mm
We have seen that the matrix $D$ is indeed a block lower-triangular matrix; thus to finish the proof, it is enough to show that each $D_i$ is invertible.  To see this, note that it is of the form $-I+D_i'$ for some real skew-symmetric matrix $D_i'$.  Since the nonzero eigenvalues of a real skew-symmetric matrix are purely imaginary, we see that $D_i$ cannot have any eigenvalues equal to zero and hence is invertible.  It follows that the matrix $D$ is invertible.  \end{proof}

We are now in position to prove the first part of Theorem \ref{intro-bound}.

\begin{thm} Let $C$ be an $m\times n$ Cauchon diagram with $n\le m$.  Then ${\rm dim}(\ker(M(C)))\le n$. As a consequence the dimension of the $\hc$-stratum associated to $J_C$ is at most $n$. \label{thm: upper}
\end{thm}
\begin{proof} Assume first that $C$ has no all black columns. By Lemma \ref{lem: 1}, there exists an invertible matrix $S$ such that by deleting $n$ rows and $n$ columns from $S \cdot M(C)$ is invertible. This means that ${\rm dim}(\ker(S\cdot M(C))) \leq n$. Moreover ${\rm dim}(\ker(M(C)))= {\rm dim}(\ker(S\cdot M(C)))$, since $S$ is invertible.  The result follows in this case.

Assume now that $C$ has at least one all black column. Let $\widehat{C}$ be the Cauchon diagram obtained from $C$ by removing every all black column of $C$. Then $M(C)=M(\widehat{C})$, and so the result follows from the previous case.
 \end{proof}

We can now prove Theorem \ref{intro-bound}.
\begin{thm}
\label{notintro-bound}
Let $m$ and $n$ be natural numbers.  Then the dimensions of $\hc$-strata in $\oqmmn$ are all at most $\min(m,n)$; moreover, for each $d\in \{0,1,\ldots, \min(m,n)\}$ there exists a $d$-dimensional $\hc$-stratum.
\end{thm}
\begin{proof} We assume that $n\le m$.   By Theorem \ref{thm: upper}, the dimensions of the $\hc$-strata are all at most $n$ and so it is sufficient to show that each of these values can occur.
Let $d\le n$.  We take $P$ to be the $\hc$-prime corresponding to the $m\times n$ Cauchon diagram whose $(i,j)$ square is white if and only if $i$ and $j$ are both at most $d$.
Then $\oqmmn/P$ is isomorphic to the ring of $d\times d$ quantum matrices.
It follows from \cite[Theorem 2.5]{laulen} that the dimension of the stratum associated to $P$ is exactly $d$.  This completes the proof.
\end{proof}

\subsection{Proof of Theorem \ref{intro2}}
In this section we use Theorem \ref{notintro-increasing} along with Theorem \ref{notintro-bound} to prove Theorem \ref{intro2}.

We first make a remark that will be useful in proving the next proposition.
\begin{rem} \label{remarkparity}{\em Let $A$ be an $n\times n$ real skew-symmetric matrix.  Then the dimension of the kernel of $A$ has the same parity as $n$.}
\end{rem}

\begin{prop} 
\label{step1} Let $C$ be an $m\times n$ Cauchon diagram with $n\le m$.  Suppose that the kernel of $M(C)$ has dimension $e\ge 1$.  Then there is a Cauchon diagram $C'\supseteq C$ obtained by adding exactly one black box to $C$ such that $M(C')$ has an $(e-1)$-dimensional kernel.
\end{prop}
\begin{proof} Clearly we can assume that $C$ has no all black columns. 
Let $d$ be the number of white boxes in $C$. We make $C$ into an $m\times n$ labelled Cauchon diagram with labels $\gc 1 ,d\dc$.

Let $T=\{i_1,\ldots ,i_k\}$ denote the set of all labels of white boxes of $C$ with the property that if one of these labels is coloured black and the remaining boxes of $C$ are left unchanged then the result is still a Cauchon diagram.

For example, in Figure 3 the labels of the white boxes which can be coloured black to obtain a Cauchon diagram are $1,3,4,8$ and $15$.

Given $i\in T$, we let $C_i$ denote the Cauchon diagram obtained by colouring the white box with label $i$ black.  

If ${\rm dim}({\rm ker}(M(C_i)))\ge e$ for every $i\in T$, then by parity considerations (see Remark \ref{remarkparity}) we must have
$${\rm dim}({\rm ker}(M(C_i)))\ge e+1,$$ for $i\in T$.  Let $$\{v_1^{(i)},\ldots ,v_{e+1}^{(i)}\}$$ be a linearly independent set of vectors in the kernel of $M(C_i)$.  We construct $e+1$ vectors as follows.  For $1\le j\le e+1$, let $w_j^{(i)}$ be the $d\times 1$ column vector whose $\ell$th coordinate is the $\ell$th coordinate of $v_j^{(i)}$ if $\ell < i$, is $0$ if $\ell=i$, and is the $(\ell-1)$ coordinate of $v_j^{(i)}$ if $\ell>i$.  By construction, every row of $M(C)$ is orthogonal to the linearly independent set $\{w_1^{(i)},\ldots ,w_{e+1}^{(i)}\}$ except for possibly the $i$th row.  Let $r$ denote the $i$th row of $M(C)$.  Then $r$ is an $1\times d$ row vector.  We then have a map 
$${\rm Span}(\{w_1^{(i)},\ldots ,w_{e+1}^{(i)}\})\rightarrow \mathbb{Q}$$ in which a vector $w$ in the span is sent to $r\cdot w\in \mathbb{Q}$. This map is surjective since otherwise the dimension of the kernel of $M(C)$ would be at least $(e+1)$-dimensional. Thus the kernel of this map is a $e$-dimensional subspace of $${\rm Span}(\{w_1^{(i)},\ldots ,w_{e+1}^{(i)}\})$$ which lies in the kernel of $M(C)$.  Since the kernel of $M(C)$ is exactly $e$-dimensional, and every vector in ${\rm Span}(\{w_1^{(i)},\ldots ,w_{e+1}^{(i)}\})$ has a zero in the $i$th coordinate, we see that every vector in ${\rm ker}(M(C))$ has a zero in the $i$th coordinate; moreover, this is the case for every $i\in T$.   For $1\le j\le n$, recall that $a_j$ denotes the smallest label which appears in column $j$ of $C$.  Then $\{a_1,\ldots ,a_n\}\subseteq T$, as they are the labels of the upper-most white boxes in each column.  By Lemma \ref{lem: 1} there is a lower-triangular matrix $S$ such that
$S\cdot M(C)$ has the property that if columns $a_1,\ldots ,a_n$ and rows $a_1,\ldots ,a_n$ are deleted then the resulting matrix is invertible.  Let $\{v_1,\ldots ,v_e\}$ be a basis for ${\rm ker}(M(C))$.  For each $i\le e$, let $u_i$ denote the $(d-n)\times 1$ column vector obtained by taking the $d\times 1$ vector $v_i$ and simply removing coordinates $a_1$ through $a_n$.  Since the $a_j$th coordinate of $v_i$ is $0$ for each $i$ and $j$, we see that $u_1,\ldots ,u_e$ are linearly independent and are in the kernel of the matrix obtained by removing columns $a_1,\ldots ,a_n$ from $S\cdot M(C)$.  Thus these vectors are in the kernel of the matrix obtained by deleting columns $a_1,\ldots ,a_n$ and rows $a_1,\ldots ,a_n$ of $S\cdot M(C)$.  But this contradicts the fact that the matrix obtained by deleting columns $a_1,\ldots ,a_n$ and rows $a_1,\ldots ,a_n$ of $S\cdot M(C)$ is invertible.  It follows that there is some $i\in T$ such that the Cauchon diagram $C_i$ has ${\rm dim}({\rm ker}(M(C_i))\le e-1$. We claim that ${\rm dim}({\rm ker}(M(C_i))= e-1$. To see this, observe that it is no loss of generality to assume that in the basis 
$\{v_1,\ldots ,v_e\}$ of ${\rm ker}(M(C))$ that the vectors $v_1, \ldots, v_{e-1}$ have a zero $i$th coordinate. Then the vectors $v_1', \ldots, v_{e-1}'$ obtained by deleting the $i$th coordinate from $v_1, \ldots, v_{e-1}$ are in the kernel of $M(C_i)$ and are linearly independent. The result follows. \end{proof} 

Recall from \cite[Theorem II.8.4]{bgbook} that an $\hc$-prime ideal in $R=\oqmmn$ is primitive if and only if its associated $\hc$-stratum is $0$-dimensional.
\begin{thm}
\label{chain}
Let $J_C$ be the $\hc$-prime of $R=\oqmmn$ associated to the Cauchon diagram $C$. Suppose that the dimension of the $\hc$-stratum associated to $J_C$ is equal to $e$. Then there is a chain $C=C_0 \subsetneq C_1 \subsetneq \dots \subsetneq C_e$ of $m\times n$ Cauchon diagrams such that:
\begin{itemize}
\item $J_C  =J_{C_0} \subsetneq J_{C_1} \subsetneq \dots \subsetneq J_{C_e}$; and
\item for all $i$ the dimension of the $\hc$-stratum associated to $J_{C_i}$ is $e-i$. 
\end{itemize}
In particular, $J_{C_e}$ is a primitive $\hc$-prime ideal in $\oqmmn$.
\end{thm}
\begin{proof} We prove this by induction on $e$.  If $e=0$, then $J_C$ is primitive and there is nothing to prove.  

Suppose now that $e >0$. By Proposition \ref{step1}, there exists a Cauchon diagram $C_1$ obtained by turning a single white box of $C$ black such that
${\rm ker}(M(C_1))$ has dimension $e-1$. As $C=C_0 \subsetneq C_1$, it follows from Theorem \ref{notintro-increasing} that $J_C  =J_{C_0} \subsetneq J_{C_1}$. 
Moreover it follows from Proposition \ref{dimstratum} that the dimension of the $\hc$-stratum associated to $J_{C_1}$ is equal to $\dim {\rm ker}(M(C_1)) =e-1$. The result is obtained by applying the induction hypothesis to $J_{C_1}$.
\end{proof}

\section{A conjecture}
Our results show that the possible dimensions of the $\hc$-strata of $\oqmmn$ are $\{0,1,2,\ldots, \min(m,n)\}$; moreover, every one of these values occurs as the dimension of some $\hc$-stratum.  What is still unresolved, however, is how exactly the dimensions of $\hc$-strata of $\oqmmn$ are distributed.  Earlier, we investigated the enumeration of $0$-dimensional $\hc$-strata in $\oqmmn$ \cite{BLN,BLL}; this is the same as enumerating the primitive $\hc$-primes of $\oqmmn$.  We make a general conjecture about the proportion of $i$-dimensional $\hc$-strata in $\oqmmn$ with $i\le m$ and $n\ge m$.
\begin{conj} Let $m$ be a natural number.  Then for $0\le i\le m$ we have
$$\lim_{n\rightarrow\infty} \frac{ \# i\mbox{-}{ \rm dimensional~} \hc\mbox{-}{\rm strata ~in ~ } \oqmmn}{ \# \hc  \mbox{-}{\rm strata ~in ~ } \oqmmn} \ = \ 2^{1-\delta_{i,0}} {2m \choose m+i}4^{-m}$$ as $n\rightarrow \infty.$
\end{conj}
We have shown this for $(i,m)\in \{(0,1),(1,1),(0,2),(1,2),(2,2),(0,3),(2,3)\}$; moreover, extensive computer computations suggest this is true in general.\\

\begin{flushleft}
\textbf{Acknowledgments.} We thank the anonymous referee for his/her  comments; they have greatly improved this text.
\end{flushleft}

\newpage

 \end{document}